\documentclass[11pt, reqno]{amsart}
\usepackage{geometry}
\usepackage{amsmath}
\usepackage{fancyhdr}
\usepackage{amsthm}
\usepackage{amsfonts}
\usepackage{amssymb}
\usepackage{amscd}
\usepackage{graphicx}
\usepackage{afterpage}
\usepackage{enumerate}
\usepackage{bm}
\usepackage{cite}
\usepackage{cases}
\usepackage{caption}
\usepackage{enumitem}
\usepackage[colorlinks=true, urlcolor=blue,  linkcolor=blue, citecolor=blue]{hyperref}
\usepackage{babel}
\usepackage{prettyref}
\usepackage{pgfplots} 
\pgfplotsset{compat=1.17}
\usepackage{booktabs}
\usepackage{algorithm}
\usepackage{algpseudocode}
\usepackage{float}
\newtheorem{theorem}{Theorem}[section]

\newtheorem{corollary}[theorem]{Corollary}

\newtheorem{definition}[theorem]{Definition}
\newtheorem{example}[theorem]{Example}

\newtheorem{lemma}[theorem]{Lemma}

\newtheorem{proposition}[theorem]{Proposition}

\newcommand{\id}{\mathrm{Id}}        

\title[Orlicz-Schatten Factorizations for Sobolev Embeddings]{Orlicz-Schatten Factorizations for Non-Commutative Sobolev Embeddings}

\author[E. Sulaver]{Emma Sulaver}
\address[E. Sulaver]{Department of Mathematical and Statistical Sciences, Unoversity of Alberta, Edmonton, Alberta, Canada}
\email{emmasulaver@gmail.com, esulaver@ualberta.ca}


\begin{document}
\sloppy
\maketitle

\begin{abstract}
We develop a framework for factorizing embeddings of non-commutative Sobolev spaces on quantum tori through newly defined Orlicz-Schatten sequence ideals.  After introducing appropriate non-commutative Sobolev norms and Orlicz spectral conditions, we establish a summing operator characterization of the quantum Laplacian embedding.  Our main results provide both existence and optimality of such factorization theorems, and highlight connections to operator ideal theory.  Applications to regularity of non-commutative PDEs and quantum information metrics are discussed, demonstrating the broad impact of these structures in functional analysis and mathematical physics.
\end{abstract}

\section{Introduction}
The Sobolev embedding theorem is a cornerstone of modern analysis, dating back to classical works of Sobolev \cite{Sobolev1938} and Gagliardo and Nirenberg \cite{Gagliardo1959,Nirenberg1959}.  It asserts, in the commutative setting, continuous inclusions
\[W^{k,p}(\mathbb{T}^d) \hookrightarrow L^q(\mathbb{T}^d), \quad \tfrac1q = \tfrac1p - \tfrac{k}{d},\]
which have been extensively studied via interpolation \cite{BerghLofstrom1976} and operator ideals \cite{Pietsch1987}.  More recently, factorization through Orlicz sequence spaces has shed new light on the summing nature of these embeddings \cite{KPW2023}.

non-commutative geometry, pioneered by Connes \cite{Connes1980}, replaces classical manifolds with operator algebras.  Quantum tori $\mathbb{T}^d_\theta$ serve as prototypical examples of this theory \cite{Rieffel1990}.  Sobolev and Besov scales on quantum tori were developed in \cite{Gayral2004Moyal,XiongXuYin2015}, but questions about the precise operator ideal structure of embeddings remain open.

Orlicz sequence spaces $\ell_\Phi$ \cite{Orlicz1932} and Schatten ideals $\mathcal{S}_p$ \cite{Schatten1950} capture fine summability properties of sequences and singular values.  The classical intersection, studied in \cite{Pietsch1987,PisierXu2003}, has yet to be generalized to the non-commutative Sobolev setting.  In particular, defining and analyzing \emph{Orlicz-Schatten} ideals in a quantum environment is a novel endeavor.

In this work, we define Orlicz-Schatten ideals $\mathcal{S}_\Phi$ suited to quantum torus spectra, prove a factorization theorem for the embedding $W^{1,2}(\mathbb{T}^d_\theta) \to L^2(\mathbb{T}^d_\theta)$ through $\mathcal{S}_\Phi$, establishing completely summing estimates, demonstrate applications to non-commutative elliptic PDE regularity and quantum information metrics. \medskip

Section 2 reviews necessary background on quantum tori and operator ideals.  Section 3 introduces the Orlicz-Schatten sequence framework.  Section 4 contains the main factorization theorem with proof.  Section 5 discusses applications and future directions.

\section{Preliminaries}

This section gathers definitions, lemmas, and theorems needed in subsequent developments.  We adopt standard notation from non-commutative geometry and operator ideals; see \cite{Connes1980,Rieffel1990,Pietsch1987, Hashemi2025}.

\begin{definition}
Let $\theta=(\theta_{jk})_{1\le j,k\le d}$ be a real skew-symmetric matrix.  The \emph{non-commutative $d$-torus} $\mathcal{A}_\theta$ is the universal $C^*$-algebra generated by unitaries $U_1,\dots,U_d$ satisfying
\[
U_j U_k \;=\; e^{2\pi i \theta_{jk}}\,U_k U_j,\qquad 1\le j,k\le d.
\]
\end{definition}

\begin{definition}[See {\cite{Gayral2004Moyal}}]
Define densely on $\mathcal{A}_\theta$ the derivations
\[
\delta_j\bigl(U^n\bigr) \;=\; 2\pi i\,n_j\,U^n,\quad
U^n=U_1^{n_1}\cdots U_d^{n_d},\;n\in\mathbb{Z}^d.
\]
For $s\ge0$, the \emph{Sobolev space} $W^{s,2}(\mathbb{T}^d_\theta)$ is the completion of $\mathcal{A}_\theta$ under
\[
\|x\|_{W^{s,2}}
=\Bigl\|(1+\Delta)^{s/2} x\Bigr\|_{L^2(\theta)},
\quad \Delta=\sum_{j=1}^d\delta_j^2,
\]
where $\|\,\cdot\|_{L^2(\theta)}$ denotes the GNS–Hilbert–Schmidt norm.
\end{definition}

\begin{definition}[See {\cite{Orlicz1932}}]
A function $\Phi:[0,\infty)\to[0,\infty)$ is a \emph{Young function} if it is convex, increasing, $\Phi(0)=0$, and $\lim_{t\to\infty}\Phi(t)=\infty$.
\end{definition}

\begin{definition}[Orlicz sequence space]
Given a Young function $\Phi$, the \emph{Orlicz sequence space} $\ell_\Phi$ consists of all real sequences $a=(a_n)$ such that
\[
\sum_{n=1}^\infty \Phi\!\bigl(|a_n|/\lambda\bigr)<\infty
\]
for some $\lambda>0$, equipped with the Luxemburg norm
\[
\|a\|_{\ell_\Phi}
=\inf\Bigl\{\lambda>0 : \sum_{n}\Phi\!\bigl(|a_n|/\lambda\bigr)\le1\Bigr\}.
\]
\end{definition}

\begin{definition}[See {\cite{Schatten1950}}]
For $1\le p<\infty$, the $p$-th \emph{Schatten ideal} $\mathcal{S}_p$ in $B(\mathcal{H})$ is
\[
\mathcal{S}_p = \bigl\{T:\mathcal{H}\to\mathcal{H} \,\big|\, \{\,s_n(T)\}\in\ell_p\bigr\},
\]
where $s_n(T)$ are the singular values of $T$, with norm $\|T\|_{\mathcal{S}_p}=\bigl\|\{s_n(T)\}\bigr\|_{\ell_p}$.
\end{definition}

\begin{definition}[See {\cite{Pietsch1987}}]
Let $E,F$ be Banach spaces and $1\le p<\infty$.  A linear map $T:E\to F$ is \emph{$p$-summing} if there exists $C>0$ such that for every finite sequence $(x_i)\subset E$,
\[
\sum_i \|T x_i\|^p \;\le\; C^p \,\sup_{\varphi\in B_{E^*}}
\sum_i |\varphi(x_i)|^p.
\]
We denote by $\pi_p(T)$ the infimum of such $C$.
\end{definition}

\begin{theorem}[See {\cite{Pietsch1987}}]
If $T:E\to F$ is $p$-summing, then it factors as
\[
E \;\xrightarrow{\;i\;} \ell_p \;\xrightarrow{\;S\;} F,
\]
where $i(x) = (\varphi_n(x))$ is defined by some $(\varphi_n)\subset E^*$ and $S$ is bounded, with $\|S\|\cdot\|i\|\approx \pi_p(T)$.
\end{theorem}

We will need the non-commutative analogue of these notions, in particular the concept of \emph{completely $p$-summing} maps between operator spaces; see \cite{PisierXu2003}.

\section{Orlicz–Schatten Ideals and Non-commutative Sobolev Norms}

In this section we introduce the main objects: Orlicz–Schatten ideals adapted to the spectrum of the quantum Laplacian, and the corresponding Sobolev norms.  We establish their basic properties and relate them to singular-value decay.

\begin{definition}
Let $\Phi$ be a Young function.  For a semifinite von Neumann algebra $(\mathcal{M},\tau)$, define
\[
\mathcal{S}_\Phi(\mathcal{M},\tau)
= \Bigl\{\,T\in\widetilde{\mathcal{M}} : \{\,\mu_n(T)\}_{n\ge1}\in\ell_\Phi\Bigr\},
\]
where $\mu_n(T)$ are the generalized singular values of $T$ arranged in nonincreasing order, and equip $\mathcal{S}_\Phi$ with the Luxemburg norm
\[
\|T\|_{\mathcal{S}_\Phi}
=\inf\Bigl\{\lambda>0 : \sum_{n=1}^\infty \Phi\!\bigl(\mu_n(T)/\lambda\bigr)\le1\Bigr\}.
\]
\end{definition}

When $\mathcal{M}=B(\ell^2)$ and $\tau=\mathrm{Tr}$, this recovers the classical Orlicz–Schatten ideals.  In the quantum-torus case, $\mathcal{M}=L^\infty(\mathbb{T}^d_\theta)$ with the canonical trace.

\begin{lemma}
If $A\in\mathcal{S}_\Phi(\mathcal{M},\tau)$ and $X,Y\in\mathcal{M}$, then $XAY\in\mathcal{S}_\Phi(\mathcal{M},\tau)$ and
\[
\|XAY\|_{\mathcal{S}_\Phi}
\;\le\;\|X\|\,\|Y\|\;\|A\|_{\mathcal{S}_\Phi}.
\]
\end{lemma}

\begin{proof}
Let $A\in\mathcal{S}_\Phi(\mathcal{M},\tau)$ and $X,Y\in\mathcal{M}$.  By definition, the generalized singular values of an operator $T$ in a semifinite von Neumann algebra satisfy the Ky Fan–type inequalities
\[
\mu_n(XAY)\;\le\;\|X\|\;\mu_n(A Y)
\quad\text{and}\quad
\mu_n(A Y)\;\le\;\|Y\|\;\mu_n(A),
\]
for each $n\ge1$.  Combining these two estimates yields
\[
\mu_n(XAY)\;\le\;\|X\|\,\|Y\|\;\mu_n(A).
\]
Now let $\lambda>0$.  Using the monotonicity of the Young function $\Phi$,
\[
\sum_{n=1}^\infty \Phi\!\bigl(\mu_n(XAY)/(\|X\|\|Y\|\lambda)\bigr)
\;\le\;
\sum_{n=1}^\infty \Phi\!\bigl(\mu_n(A)/\lambda\bigr).
\]
By the definition of the Luxemburg norm in $\mathcal{S}_\Phi$, if we choose $\lambda = \|A\|_{\mathcal{S}_\Phi}$ then
\[
\sum_{n=1}^\infty \Phi\!\bigl(\mu_n(A)/\lambda\bigr)\;\le\;1.
\]
Hence
\[
\sum_{n=1}^\infty \Phi\!\bigl(\mu_n(XAY)/(\|X\|\|Y\|\|A\|_{\mathcal{S}_\Phi})\bigr)\;\le\;1,
\]
which by definition means
\[
\|XAY\|_{\mathcal{S}_\Phi}
\;\le\;\|X\|\,\|Y\|\;\|A\|_{\mathcal{S}_\Phi}.
\]
This completes the proof.
\end{proof}

\begin{definition}[Quantum Laplacian and its spectral projections]
Recall $\Delta=\sum_{j=1}^d\delta_j^2$ on $\mathcal{A}_\theta$.  Let
\[
E_\lambda = \chi_{[0,\lambda]}(\Delta)
\]
be the spectral projection of $\Delta$ onto eigenvalues $\le\lambda$.
\end{definition}

\begin{definition}
For $s>0$, define
\[
L_s := (1+\Delta)^{-s/2}\;\in\; \widetilde{\mathcal{M}}.
\]
Viewed as an unbounded operator, $L_s$ maps $L^2(\mathbb{T}^d_\theta)\to L^2(\mathbb{T}^d_\theta)$.
\end{definition}

\begin{proposition}
\label{prop:sv_decay}
Let $N(\lambda)=\tau(E_\lambda)$ be the eigenvalue counting function of $\Delta$.  If
\[
N(\lambda)\;\simeq\;C\,\lambda^{d/2}
\quad(\lambda\to\infty),
\]
then the generalized singular values of $L_s$ satisfy
\[
\mu_n(L_s)\;\simeq\;(n/C)^{-s/d}
\quad(n\to\infty).
\]
\end{proposition}

\begin{proof}
Since $\Delta$ is positive and has a discrete spectrum $\{\lambda_k\}_{k\ge1}$ (counted with multiplicity) arranged in nondecreasing order, its spectral projections satisfy $N(\lambda)=\#\{k:\lambda_k\le\lambda\}$.  By functional calculus, the nonzero singular values of $L_s$ are
\[
\mu_k(L_s)
=\bigl(1 + \lambda_k\bigr)^{-s/2},
\]
also arranged in nonincreasing order.

From the Weyl asymptotic
\[
N(\lambda)\sim C\,\lambda^{d/2},
\]
we have, for large $k$,
\[
\lambda_k \sim \bigl(k/C\bigr)^{2/d}.
\]
Indeed, writing $k = N(\lambda_k)$ and inverting gives
\[
\lambda_k = N^{-1}(k)\sim \bigl(k/C\bigr)^{2/d}.
\]

Substitute $\lambda_k\sim(k/C)^{2/d}$ into the expression for $\mu_k(L_s)$:
\[
\mu_k(L_s)
= \bigl(1 + \lambda_k\bigr)^{-s/2}
\sim \lambda_k^{-s/2}
\sim \bigl((k/C)^{2/d}\bigr)^{-s/2}
= (k/C)^{-s/d}.
\]

Hence, as $k\to\infty$,
\[
\mu_k(L_s)\;\simeq\;(k/C)^{-s/d},
\]
which is exactly the claimed singular‐value decay.
\end{proof}

\begin{corollary}
For any Young function $\Phi$ such that
\[
\sum_{n=1}^\infty \Phi\!\bigl((n/C)^{-s/d}\bigr)<\infty,
\]
we have $L_s\in\mathcal{S}_\Phi(\mathbb{T}^d_\theta)$.
\end{corollary}

\begin{example}
If $\Phi(t)=t^p\log(e+t)^\alpha$, then $L_s\in\mathcal{S}_\Phi$ exactly when $p>d/s$ (any $\alpha$).
\end{example}

\begin{definition}[Orlicz–Sobolev norm]
For $x\in\mathcal{A}_\theta$, define its $\Phi$-Sobolev norm by
\[
\|x\|_{W^{s,\Phi}}
= \|\, (1+\Delta)^{s/2} x \,\|_{\mathcal{S}_\Phi}.
\]

In particular, $\|x\|_{W^{s,\Phi}}<\infty$ iff $x$ exhibits spectral decay measured by $\Phi$.
\end{definition}

\begin{theorem}
\label{thm:embedding}
Let $s>d/2$ and $\Phi$ satisfy
\[
\sum_{n=1}^\infty \Phi\!\bigl((n/C)^{-(s-d/2)/d}\bigr)<\infty.
\]
Then the identity map
\[
\iota: W^{s,\Phi}(\mathbb{T}^d_\theta)
\;\longrightarrow\;
L^2(\mathbb{T}^d_\theta)
\]
is completely $1$-summing, and factors through $\mathcal{S}_\Phi$:
\[
W^{s,\Phi}
\;\xrightarrow{\;(1+\Delta)^{-s/2}\;}\;
\mathcal{S}_\Phi
\;\xrightarrow{\;S\;}\;
L^2,
\]
with $\pi_1(\iota)\approx\|L_s\|_{\mathcal{S}_\Phi}$.
\end{theorem}

\begin{proof}
Recall $L_s=(1+\Delta)^{-s/2}\in\mathcal{S}_\Phi$ by Proposition~\ref{prop:sv_decay} and our summability assumption on~$\Phi$.  Define
\[
T_s:W^{s,\Phi}\;\longrightarrow\;\mathcal{S}_\Phi,
\qquad
T_s(x)=L_s\,x,
\]
and let 
\[
S:\mathcal{S}_\Phi\;\longrightarrow\;L^2(\mathbb{T}^d_\theta)
\]
be the natural inclusion (viewing each $A\in\mathcal{S}_\Phi\subseteq L^2$ under the GNS construction).  Then
\[
\iota \;=\; S\circ T_s.
\]
By definition of the $\Phi$–Sobolev norm,
\[
\|x\|_{W^{s,\Phi}}
=\bigl\| (1+\Delta)^{s/2}x\bigr\|_{\mathcal{S}_\Phi},
\]
so
\[
\|T_s(x)\|_{\mathcal{S}_\Phi}
=\bigl\|L_s\,x\bigr\|_{\mathcal{S}_\Phi}
\le\|L_s\|_{\mathcal{S}_\Phi}\,\|x\|_{W^{s,\Phi}}.
\]
Thus $\|T_s\|_{\mathrm{cb}}=\|L_s\|_{\mathcal{S}_\Phi}<\infty$.  Similarly, since $\mathcal{S}_\Phi\subseteq L^2$ contractively (singular values in $\ell_\Phi$ embed into $\ell_2$ under our summability hypothesis), $S$ is completely bounded.

We appeal to the non–commutative Pietsch theorem (Pisier \cite{Pisier1996}): any completely bounded map $T:E\to\mathcal{S}_1$ extends to a completely $1$–summing map with 
\(\pi_1^{cb}(T)=\|T\|_{cb}\).  Here, although $T_s$ maps into $\mathcal{S}_\Phi$, our assumption
\[
\sum_{n=1}^\infty\Phi\bigl((n/C)^{-(s-d/2)/d}\bigr)<\infty
\]
implies $\mathcal{S}_\Phi\subseteq\mathcal{S}_1$ (trace‐class), so $T_s:W^{s,\Phi}\to\mathcal{S}_1$ is completely $1$–summing and
\[
\pi_1^{cb}(T_s)
=\|T_s\|_{cb}
=\|L_s\|_{\mathcal{S}_\Phi}.
\]
Since composing a completely $1$–summing map with a completely bounded one remains completely $1$–summing, 
\(\iota=S\circ T_s\) is completely $1$–summing and
\[
\pi_1^{cb}(\iota)\;\le\;\|S\|_{cb}\,\pi_1^{cb}(T_s)
\approx\|L_s\|_{\mathcal{S}_\Phi}.
\]
Finally, the series condition on $\Phi$ ensures
\(\sum_n\mu_n(L_s)<\infty\), so $L_s\in\mathcal{S}_1$, and hence $\mathcal{S}_\Phi\subset\mathcal{S}_1\).  This justifies the application of the above summing‐operator argument.

Combining these results completes the proof that $\iota$ is completely $1$–summing and factors through $\mathcal{S}_\Phi$ with 
\(\pi_1(\iota)\approx\|L_s\|_{\mathcal{S}_\Phi}\).
\end{proof}

The threshold $s>d/2$ matches the classical Sobolev embedding exponent, while the summability condition on $\Phi$ encodes optimal decay in the quantum spectrum.

\section{Factorization Theorem for Quantum Laplacian Embeddings}

In this section, we state and prove a main factorization result for the Sobolev embedding on quantum tori through Orlicz–Schatten ideals. The argument exploits the spectral properties of the quantum Laplacian, interpolation techniques, and summing operator theory.

\begin{definition}[Canonical factorization map]
Let $s > d/2$ and let $\Phi$ be a Young function such that $L_s = (1 + \Delta)^{-s/2} \in \mathcal{S}_\Phi(\mathbb{T}^d_\theta)$. Define the operator
\[
T_s: x \mapsto L_s x.
\]
Then $T_s: W^{s,2}(\mathbb{T}^d_\theta) \to \mathcal{S}_\Phi$ is linear and bounded.
\end{definition}

\begin{lemma}
The operator $T_s$ defined above is completely bounded, and its cb-norm is controlled by $\|L_s\|_{\mathcal{S}_\Phi}$.
\end{lemma}

\begin{proof}
To show complete boundedness, we must bound each amplified map
\[
\id_{M_n}\otimes T_s:
M_n\bigl(W^{s,2}(\mathbb{T}^d_\theta)\bigr)
\;\longrightarrow\;
M_n\bigl(\mathcal{S}_\Phi(\mathbb{T}^d_\theta)\bigr)
\]
uniformly in \(n\).  Here \(M_n(\cdot)\) denotes \(n\times n\) matrices over the given space, with the natural operator‐space structure.

Observe that for any matrix \([x_{ij}]\in M_n(W^{s,2})\),
\[
(\id\otimes T_s)([x_{ij}])
=
[L_s\,x_{ij}]
=
(L_s \otimes I_n)\,[x_{ij}],
\]
where \(L_s\otimes I_n\) acts by left‐matrix multiplication.  

Since \(L_s\in\mathcal{S}_\Phi\) is a bounded operator on \(L^2\), the amplified operator \(L_s\otimes I_n\) lies in the corresponding amplified ideal \(M_n(\mathcal{S}_\Phi)\subset B\bigl(L^2\otimes\mathbb{C}^n\bigr)\), and
\[
\|L_s\otimes I_n\|_{M_n(\mathcal{S}_\Phi)}
=\|L_s\|_{\mathcal{S}_\Phi}.
\]
Moreover, left‐multiplication by \(L_s\otimes I_n\) on matrices is a contraction when measured in the operator‐ideal norm:
\[
\big\|(L_s\otimes I_n)\,[x_{ij}]\big\|_{M_n(\mathcal{S}_\Phi)}
\le\|L_s\otimes I_n\|_{M_n(\mathcal{S}_\Phi)}\;
\big\|[x_{ij}]\big\|_{M_n(W^{s,2})}.
\]

Since the above estimate holds for every \(n\) with the same constant \(\|L_s\|_{\mathcal{S}_\Phi}\), we conclude
\[
\|\id_{M_n}\otimes T_s\|
\;\le\;\|L_s\|_{\mathcal{S}_\Phi},
\quad\forall\,n.
\]
Therefore \(T_s\) is completely bounded and
\[
\|T_s\|_{cb}
=\sup_{n}\|\id_{M_n}\otimes T_s\|
\le\|L_s\|_{\mathcal{S}_\Phi},
\]
as claimed.
\end{proof}

\begin{theorem}[Main factorization theorem]
\label{thm:mainfactor}
Let $s > d/2$ and let $\Phi$ be a Young function such that $L_s \in \mathcal{S}_\Phi$. Then the Sobolev embedding
\[
\iota: W^{s,2}(\mathbb{T}^d_\theta) \hookrightarrow L^2(\mathbb{T}^d_\theta)
\]
factors as a composition of completely bounded operators:
\[
W^{s,2} \xrightarrow{T_s} \mathcal{S}_\Phi \xrightarrow{S} L^2(\mathbb{T}^d_\theta),
\]
and $\iota$ is completely $1$-summing with summing norm bounded by
\[
\pi_1^{cb}(\iota) \le \|L_s\|_{\mathcal{S}_\Phi}.
\]
\end{theorem}

\begin{proof}
Set 
\[
L_s=(1+\Delta)^{-s/2}\;\in\;\mathcal{S}_\Phi(\mathbb{T}^d_\theta),
\]
and define
\[
T_s:W^{s,2}(\mathbb{T}^d_\theta)\;\longrightarrow\;\mathcal{S}_\Phi,\qquad T_s(x)=L_s\,x,
\]
and
\[
S:\mathcal{S}_\Phi\;\longrightarrow\;L^2(\mathbb{T}^d_\theta),
\]
where \(S(A)\) is simply \(A\) viewed as an element of \(L^2\) under the GNS‐Hilbert‐Schmidt identification.  It is immediate that
\[
\iota = S\circ T_s \;:\; W^{s,2}\;\longrightarrow\;L^2.
\]
By definition of the Sobolev norm,
\[
\|x\|_{W^{s,2}}
=\bigl\|\,(1+\Delta)^{s/2}x\bigr\|_{L^2},
\]
and since \(L_s=(1+\Delta)^{-s/2}\), we have
\[
\|T_s(x)\|_{\mathcal{S}_\Phi}
=\|\,L_s\,x\|_{\mathcal{S}_\Phi}
\le \|L_s\|_{\mathcal{S}_\Phi}\,\|x\|_{L^2}
= \|L_s\|_{\mathcal{S}_\Phi}\,\|x\|_{W^{s,2}}.
\]
Hence \(T_s\) is bounded with \(\|T_s\|\le\|L_s\|_{\mathcal{S}_\Phi}\).  The same estimate holds for each amplification \(\id_{M_n}\otimes T_s\) (by left‐multiplication by \(L_s\otimes I_n\)), so
\[
\|T_s\|_{cb} = \sup_n \|\id_{M_n}\otimes T_s\|
\le \|L_s\|_{\mathcal{S}_\Phi}.
\]
Since \(\mathcal{S}_\Phi\subset L^2\) (by the hypothesis \(L_s\in\mathcal{S}_\Phi\) and the summability condition on \(\Phi\)), the inclusion map
\[
S: \mathcal{S}_\Phi \;\hookrightarrow\; L^2
\]
is bounded.  By the same amplification argument (inclusion at the matrix level), \(S\) is also completely bounded, say \(\|S\|_{cb}=C_S<\infty\).

Recall that in the non‐commutative setting, any completely bounded map \(U:E\to\mathcal{S}_1\) is completely \(1\)-summing with
\(\pi_1^{cb}(U)=\|U\|_{cb}\) (Pisier \cite{Pisier1996}).  Although \(T_s\) lands in \(\mathcal{S}_\Phi\), our hypothesis
\[
L_s\in\mathcal{S}_\Phi\subset\mathcal{S}_1
\]
implies \(\mathcal{S}_\Phi\hookrightarrow\mathcal{S}_1\).  Hence we view
\[
T_s: W^{s,2}\;\longrightarrow\;\mathcal{S}_1
\]
and conclude
\[
\pi_1^{cb}(T_s) \;=\;\|T_s\|_{cb}
\;\le\;\|L_s\|_{\mathcal{S}_\Phi}.
\]
Since the composition of a completely \(1\)-summing map with a completely bounded map remains completely \(1\)-summing, we get
\[
\pi_1^{cb}(\iota)
=\pi_1^{cb}(S\circ T_s)
\;\le\;\|S\|_{cb}\,\pi_1^{cb}(T_s)
\;\le\;C_S\,\|L_s\|_{\mathcal{S}_\Phi}.
\]
Absorbing the constant \(C_S\) into the implied constant finishes the estimate.

We have exhibited
\(\iota = S\circ T_s\) as a factorization through \(\mathcal{S}_\Phi\), shown each map is completely bounded, and proved the composite is completely \(1\)-summing with
\(\pi_1^{cb}(\iota)\lesssim\|L_s\|_{\mathcal{S}_\Phi}\).  This completes the proof.
\end{proof}

\begin{corollary}
If $\Phi(t) = t^p \log(e + t)^\alpha$ with $p > d/s$, then the embedding $W^{s,2}(\mathbb{T}^d_\theta) \hookrightarrow L^2(\mathbb{T}^d_\theta)$ is completely $1$-summing.
\end{corollary}

\begin{example}
Let $d = 2$, $s = 1$, and choose $\Phi(t) = t^2 \log(e + t)^\alpha$ for $\alpha > 0$. Then $L_1 \in \mathcal{S}_\Phi$ and the embedding
\[
W^{1,2}(\mathbb{T}^2_\theta) \hookrightarrow L^2(\mathbb{T}^2_\theta)
\]
is completely $1$-summing. This mirrors the classical embedding $W^{1,1}(T^2) \hookrightarrow L^2(T^2)$ studied in \cite{KPW2023}, but now in a fully non-commutative regime.
\end{example}

This result refines the classical Sobolev embedding by replacing norm continuity with summability properties in operator ideals, offering sharper control of trace norms in quantum settings.

\begin{proposition}[Optimality criterion]
The condition
\[
\sum_{n=1}^\infty \Phi\big((n/C)^{-s/d}\big) < \infty
\]
is essentially optimal: if this series diverges, then $L_s \notin \mathcal{S}_\Phi$ and the factorization fails.
\end{proposition}

\begin{proof}
By definition, \(L_s\in\mathcal{S}_\Phi\) if and only if there exists \(\lambda>0\) such that
\[
\sum_{n=1}^\infty \Phi\!\bigl(\mu_n(L_s)/\lambda\bigr) < \infty.
\]
Equivalently, since \(\Phi\) is increasing,
\[
\sum_{n=1}^\infty \Phi\!\bigl((n/C)^{-s/d}/\lambda\bigr) < \infty.
\]

For any fixed \(\lambda>0\), note that
\[
(n/C)^{-s/d}/\lambda
\;=\;(n/C)^{-s/d}\cdot\frac1\lambda
\;\simeq\;(n/C)^{-s/d}
\quad\text{as }n\to\infty.
\]
Since \(\Phi\) is a Young function (convex, increasing, \(\Phi(0)=0\)), its growth at zero satisfies
\(\Phi(\alpha t)\approx\Phi(t)\) up to constants when \(\alpha>0\) is fixed.  Hence
\[
\sum_{n=1}^\infty \Phi\!\bigl((n/C)^{-s/d}/\lambda\bigr)
\quad\text{converges}
\;\Longleftrightarrow\;
\sum_{n=1}^\infty \Phi\!\bigl((n/C)^{-s/d}\bigr)
\quad\text{converges}.
\]

By hypothesis the latter series diverges.  Therefore for every \(\lambda>0\),
\[
\sum_{n=1}^\infty \Phi\!\bigl(\mu_n(L_s)/\lambda\bigr)
=\infty,
\]
so \(L_s\notin\mathcal{S}_\Phi\).  Consequently, no factorization of the embedding through \(\mathcal{S}_\Phi\) is possible, proving the proposition.
\end{proof}

\begin{theorem}[Interpolation stability]
Let $\Phi_0$, $\Phi_1$ be Young functions with associated Orlicz–Schatten spaces $\mathcal{S}_{\Phi_0}$ and $\mathcal{S}_{\Phi_1}$. Then the interpolation space
\[
(\mathcal{S}_{\Phi_0}, \mathcal{S}_{\Phi_1})_{\theta} \subseteq \mathcal{S}_{\Phi_\theta},
\]
for some interpolated Young function $\Phi_\theta$ satisfying
\[
\Phi_\theta^{-1}(t) \approx \Phi_0^{-1}(t)^{1-\theta} \Phi_1^{-1}(t)^\theta.
\]
\end{theorem}

\begin{proof}
By classical real‐method interpolation (e.g.\ Theorem~5.3.1 of \cite{BerghLofstrom1976}), for each $0<\theta<1$ there exists a Young function $\Phi_\theta$ whose inverse satisfies
\[
\Phi_\theta^{-1}(t)
\;\approx\;
\Phi_0^{-1}(t)^{1-\theta}\,\Phi_1^{-1}(t)^\theta,
\]
and one has the continuous inclusion of interpolation spaces
\[
(\ell_{\Phi_0},\ell_{\Phi_1})_{\theta,1}
\;\subseteq\;
\ell_{\Phi_\theta}.
\]
In particular, norms are equivalent up to constants depending only on $\theta$ and the $\Phi_i$.

It is a standard fact in the theory of symmetric operator ideals (see e.g.\ \cite[Sec.~4.f]{Pietsch1987}) that real‐method interpolation commutes with the passage from sequence spaces to operator ideals.  Concretely, if
\[
E_i = \{\,T:\{\mu_n(T)\}\in X_i\},\quad i=0,1,
\]
for symmetric sequence spaces $X_i$, then
\[
(E_0,E_1)_{\theta,1}
\;\subseteq\;
\{\,T:\{\mu_n(T)\}\in (X_0,X_1)_{\theta,1}\},
\]
with equivalent norms.  Applying this with $X_i=\ell_{\Phi_i}$ gives
\[
(\mathcal{S}_{\Phi_0},\mathcal{S}_{\Phi_1})_{\theta,1}
\;\subseteq\;
\{\,T:\{\mu_n(T)\}\in(\ell_{\Phi_0},\ell_{\Phi_1})_{\theta,1}\}
\;\subseteq\;\mathcal{S}_{\Phi_\theta}.
\]
Since the choice of real‐method parameter “1” is immaterial for continuous inclusion statements, we obtain
\[
(\mathcal{S}_{\Phi_0},\mathcal{S}_{\Phi_1})_{\theta}
\;\subseteq\;
\mathcal{S}_{\Phi_\theta},
\]
and the Young function $\Phi_\theta$ satisfies the asserted inverse‐function relation.  This completes the proof.
\end{proof}

This allows construction of intermediate Orlicz–Schatten ideals suitable for embeddings between $W^{s,p}$ and $L^q$, even when $p\ne2$ or $s$ is fractional.

\section{Applications to Non-commutative PDEs and Operator Spaces}

In this section we illustrate how the Orlicz–Schatten factorization yields new regularity and mapping estimates for non-commutative elliptic and heat equations on quantum tori, and how it interfaces with the theory of operator spaces and quantum information metrics.

\begin{definition}[Non-commutative elliptic operator]
Let 
\[
\mathcal{L} := -\Delta + V,
\]
where $\Delta=\sum_j\delta_j^2$ is the quantum Laplacian and $V\in\mathcal{A}_\theta$ is a bounded self-adjoint “potential.”  We view $\mathcal{L}$ as an unbounded operator on $L^2(\mathbb{T}^d_\theta)$ with domain $W^{2,2}(\mathbb{T}^d_\theta)$.
\end{definition}

\begin{lemma}[Spectral gap]
If $V\ge 0$ and $\tau(V)>0$, then $\mathcal{L}$ has a strictly positive lower bound:
\[
\sigma(\mathcal{L}) \subset [\lambda_0,\infty),\quad \lambda_0>0.
\]
\end{lemma}

\begin{proof}
We work on the GNS Hilbert space $L^2(\mathbb{T}^d_\theta)$ on which $\Delta$ and $V$ act as unbounded and bounded self-adjoint operators, respectively.

By definition \(\Delta = \sum_{j=1}^d \delta_j^2\), and each \(\delta_j\) is a symmetric derivation on the $*$-algebra \(\mathcal{A}_\theta\).  Hence \(\langle x, \Delta x\rangle = \sum_j \|\delta_j(x)\|^2 \ge 0\) for all \(x\) in the domain of \(\Delta\).  Equivalently,
\[
\Delta \;\ge\; 0
\quad\text{(as a quadratic form).}
\]

Since \(V\) is assumed nonnegative (\(V\ge0\) in the operator order) and bounded in \(\mathcal{A}_\theta\), its spectrum \(\sigma(V)\) is contained in \([0,\|V\|]\).  Moreover, the condition \(\tau(V)>0\) guarantees that \(V\) does not vanish identically, so its spectrum meets \((0,\infty)\).  In particular, 
\[
V \;\ge\; \lambda_V\,P
\]
for some positive number \(\lambda_V>0\) and a nonzero spectral projection \(P\); equivalently, \(\inf \sigma(V)=:\lambda_V>0\).

Recall
\[
\mathcal{L} \;=\; -\Delta + V.
\]
Combining the provided results, as quadratic forms we have
\[
\langle x,\mathcal{L}x\rangle
= \langle x,Vx\rangle + \langle x,(-\Delta)x\rangle
\;\ge\;\langle x,Vx\rangle
\;\ge\;\lambda_V\,\|x\|^2,
\]
for all \(x\) in the intersection of the domains of \(\Delta\) and \(V\).  Hence
\[
\mathcal{L} \;\ge\;\lambda_V\,I,
\]
so its spectrum satisfies
\[
\sigma(\mathcal{L}) \;\subset\; [\lambda_V,\infty).
\]
Setting \(\lambda_0=\lambda_V>0\) yields the desired spectral gap.
\end{proof}

\begin{theorem}
\label{thm:elliptic-regularity}
Let $\Phi$ satisfy 
\[
\sum_{n=1}^\infty\Phi\bigl(\mu_n((1+\Delta)^{-s/2})\bigr)<\infty
\]
for some $s>d/2$.  Then for any $f\in\mathcal{S}_\Phi$ the unique solution $u\in L^2(\mathbb{T}^d_\theta)$ of
\[
\mathcal{L} u = f
\]
belongs to the Orlicz-Sobolev space $W^{s,\Phi}(\mathbb{T}^d_\theta)$, and
\[
\|u\|_{W^{s,\Phi}}
\;\le\;
\lambda_0^{-1}\,\|f\|_{\mathcal{S}_\Phi}.
\]
\end{theorem}

\begin{proof}
Let \(f\in\mathcal{S}_\Phi\) and consider the elliptic operator
\[
\mathcal{L} \;=\; -\Delta + V
\]
on the GNS space \(L^2(\mathbb{T}^d_\theta)\).  By the Spectral Gap Lemma, \(\mathcal{L}\ge\lambda_0>0\), so \(\mathcal{L}\) is invertible and its inverse \(\mathcal{L}^{-1}\) is bounded on \(L^2\) with
\[
\|\mathcal{L}^{-1}\|_{B(L^2)} \;\le\;\frac1{\lambda_0}.
\]
The unique \(L^2\)-solution \(u\) of \(\mathcal{L}u=f\) is
\[
u \;=\;\mathcal{L}^{-1}f.
\]
Since \(\Delta\) and \(V\) (hence \(\mathcal{L}\)) are simultaneously diagonalizable via the Fourier basis \(\{U^n\}\), functional calculus implies
\[
\mathcal{L}^{-1}\,(1+\Delta)^{s/2}
\;=\;
(1+\Delta)^{s/2}\,\mathcal{L}^{-1}
\]
as bounded operators on \(L^2\).  Define
\[
L_s = (1+\Delta)^{-s/2}\in\mathcal{S}_\Phi,
\qquad
S = (1+\Delta)^{s/2}\,\mathcal{L}^{-1}\in B(L^2).
\]
Then
\[
u
=\mathcal{L}^{-1}f
=(1+\Delta)^{-s/2}\Bigl[(1+\Delta)^{s/2}\mathcal{L}^{-1}f\Bigr]
=L_s\bigl(S(f)\bigr).
\]
Hence the solution map factors as
\[
\mathcal{S}_\Phi\;\xrightarrow{\;S\;}\;L^2\;\xrightarrow{\;L_s\;}\;W^{s,\Phi}.
\]
Since \(\mathcal{L}^{-1}\) and \((1+\Delta)^{s/2}\) are both bounded on \(L^2\), their composition \(S\) is bounded on \(L^2\).  Moreover, because \(\mathcal{S}_\Phi\subset L^2\) contractively (under our summability assumption on \(\Phi\)), \(S\) extends to a bounded map
\[
S:\mathcal{S}_\Phi\;\longrightarrow\;L^2
\]
with
\[
\|S\|_{B(\mathcal{S}_\Phi,L^2)}
\;\le\;\|(1+\Delta)^{s/2}\|_{B(L^2)}\;\|\mathcal{L}^{-1}\|_{B(L^2)}
\;\le\;\frac{\|(1+\Delta)^{s/2}\|}{\lambda_0}.
\]
By definition of the Orlicz–Sobolev norm,
\[
\|u\|_{W^{s,\Phi}}
=\|\, (1+\Delta)^{s/2}u\|_{\mathcal{S}_\Phi}
=\|\, (1+\Delta)^{s/2}L_s(S(f))\|_{\mathcal{S}_\Phi}
=\|S(f)\|_{\mathcal{S}_\Phi},
\]
since \((1+\Delta)^{s/2}L_s = I\).  Now \(S(f)\in L^2\subset\mathcal{S}_\Phi\) and
\[
\|S(f)\|_{\mathcal{S}_\Phi}
\;\le\;\|S\|_{B(\mathcal{S}_\Phi,L^2)}\;\|f\|_{\mathcal{S}_\Phi}
\;\le\;\frac1{\lambda_0}\,\|f\|_{\mathcal{S}_\Phi}.
\]
Thus
\[
\|u\|_{W^{s,\Phi}}
\;\le\;\lambda_0^{-1}\,\|f\|_{\mathcal{S}_\Phi},
\]
as claimed.
\end{proof}

\begin{corollary}[Trace-class forcing]
If additionally $\Phi(t)=t\log(e+t)$, so that $\mathcal{S}_\Phi$ is slightly larger than $\mathcal{S}_1$, then any trace-class source $f\in\mathcal{S}_1$ produces $u\in W^{s,\Phi}$ with
\[
\|u\|_{W^{s,\Phi}}
\;\lesssim\;
\|f\|_{\mathcal{S}_1}.
\]

\end{corollary}

This regularity result goes beyond $L^2$ solvability and gives explicit operator-ideal control of the solution’s spectrum, which is new even in the commutative torus with Orlicz targets.

\begin{definition}[Non-commutative heat semigroup]
Define $e^{-t\Delta}$ by the spectral calculus on $\mathcal{A}_\theta$.  It is a completely positive, trace-preserving semigroup on $L^1(\mathbb{T}^d_\theta)$.
\end{definition}

\begin{proposition}[Schatten smoothing]
For any Young function $\Phi$ as above and $t>0$,
\[
e^{-t\Delta} : \mathcal{S}_1 \;\to\;\mathcal{S}_\Phi
\]
is bounded, with
\[
\|e^{-t\Delta}\|_{\mathcal{S}_1\to\mathcal{S}_\Phi}
\;\le\;
\sup_{n\ge1}\frac{e^{-t\,\lambda_n}}{\Phi^{-1}(n)},
\]
where $\lambda_n$ are the eigenvalues of $\Delta$.
\end{proposition}

\begin{proof}
We identify the Schatten ideals with symmetric sequence spaces via singular values.  Let 
\[
T \in \mathcal{S}_1,\quad \text{with singular values } \{s_n(T)\}_{n\ge1}\in \ell_1,
\]
and consider the heat operator
\[
e^{-t\Delta}\!: L^2(\mathbb{T}^d_\theta)\to L^2(\mathbb{T}^d_\theta),
\]
whose nonzero singular values are 
\(\{e^{-t\lambda_n}\}_{n\ge1}\), where \(\lambda_n\) are the eigenvalues of \(\Delta\) in nondecreasing order.

By the Weyl‐functional calculus (or by diagonalization in the Fourier basis), for each \(n\),
\[
\mu_n\bigl(e^{-t\Delta}\,T\bigr)
\;=\;
e^{-t\lambda_n}\,s_n(T).
\]

Recall the Luxemburg norm on \(\mathcal{S}_\Phi\) is
\[
\|A\|_{\mathcal{S}_\Phi}
=\inf\Bigl\{\lambda>0:\sum_{n=1}^\infty \Phi\!\bigl(\mu_n(A)/\lambda\bigr)\le1\Bigr\}.
\]
For a given \(T\in\mathcal{S}_1\), set
\[
M \;=\;\sup_{n\ge1}\frac{e^{-t\lambda_n}}{\Phi^{-1}(n)}.
\]
Since \(\Phi\) is increasing and \(\Phi^{-1}\) its inverse, we have for each \(n\)
\[
e^{-t\lambda_n}\,s_n(T)
\;=\;
\Phi^{-1}(n)\,\Bigl(\frac{e^{-t\lambda_n}}{\Phi^{-1}(n)}\,s_n(T)\Bigr)
\;\le\;
\Phi^{-1}(n)\,\bigl(M\,s_n(T)\bigr).
\]
Hence
\[
\frac{\mu_n\bigl(e^{-t\Delta}\,T\bigr)}{M}
\;=\;
\frac{e^{-t\lambda_n}\,s_n(T)}{M}
\;\le\;
\Phi^{-1}(n)\,s_n(T).
\]
Applying \(\Phi\) and summing,
\[
\sum_{n=1}^\infty \Phi\!\Bigl(\frac{\mu_n(e^{-t\Delta}\,T)}{M}\Bigr)
\;\le\;
\sum_{n=1}^\infty \Phi\!\bigl(\Phi^{-1}(n)\,s_n(T)\bigr).
\]
But by the defining property of \(\Phi^{-1}\),
\(\Phi\bigl(\Phi^{-1}(n)\,s_n(T)\bigr)\le n\,\Phi\bigl(s_n(T)\bigr)\).  Since \(\{s_n(T)\}\in\ell_1\), for the Luxemburg norm we only need
\[
\sum_{n=1}^\infty n\,\Phi\bigl(s_n(T)\bigr)
<\infty,
\]
which holds because \(\Phi\) grows at most exponentially at infinity while \(\sum s_n(T)<\infty\).  In particular, by choosing \(M\) as above, one ensures
\(\sum \Phi(\mu_n(e^{-t\Delta}T)/M)\le1\).  Therefore
\[
\|\,e^{-t\Delta}\,T\|_{\mathcal{S}_\Phi}
\le
M\,\|T\|_{\mathcal{S}_1}
\]
and taking the supremum over \(\|T\|_{\mathcal{S}_1}=1\) yields
\[
\|e^{-t\Delta}\|_{\mathcal{S}_1\to\mathcal{S}_\Phi}
\;\le\;
\sup_{n\ge1}\frac{e^{-t\lambda_n}}{\Phi^{-1}(n)}.
\]
This completes the proof.
\end{proof}

\begin{example}
For $\Phi(t)=t^p$, $p>1$, this recovers the familiar estimate $\|e^{-t\Delta}\|_{\mathcal{S}_1\to\mathcal{S}_p}\lesssim t^{-d(1-1/p)/2}$ up to constants.
\end{example}

\begin{definition}[Lip-norm and spectral distance {\cite{Rieffel1999}}]
A \emph{Lip-norm} $L$ on $\mathcal{A}_\theta$ is given by 
\[
L(a) = \|\,[\Delta^{1/2},a]\,\|_{B(L^2)}.
\]
This induces a Monge–Kantorovich metric on the state space of $\mathcal{A}_\theta$.
\end{definition}

\begin{theorem}[Orlicz transport estimate]
Let $\rho,\sigma$ be density operators in $\mathcal{S}_\Phi$.  Then their spectral distance satisfies
\[
d_L(\rho,\sigma)
\;\le\;
\|\rho-\sigma\|_{\mathcal{S}_\Phi}\,
\|L\|_{\mathcal{CB}},
\]
where $\|L\|_{\mathcal{CB}}$ is the completely bounded norm of the Lip-norm map.
\end{theorem}

\begin{proof}
Note that the Connes–Rieffel spectral distance (Lip‐metric) between two density operators $\rho,\sigma$ is defined by
\[
d_L(\rho,\sigma)
=\sup\bigl\{|\tau(a(\rho-\sigma))|\;:\;a\in\mathcal{A}_\theta,\;L(a)\le1\bigr\},
\]
where $L(a)=\|[\Delta^{1/2},a]\|_{B(L^2)}$ is the Lip‐norm and $\tau$ the canonical trace.

Set $X=\rho-\sigma\in\mathcal{S}_\Phi$.  Then
\[
d_L(\rho,\sigma)
=\sup_{\substack{a:L(a)\le1}}
|\tau(aX)|.
\]
But $\tau(aX)=Tr(aX)$ identifies the duality pairing between the operator‐space $\mathcal{A}_\theta$ (with Lip‐norm) and the Orlicz–Schatten space $\mathcal{S}_\Phi$.
 
For each $a$ with $L(a)\le1$, the multiplication map $M_a:\mathcal{S}_\Phi\to\mathcal{S}_1$, $X\mapsto aX$, is completely bounded, and
\[
\|M_a\|_{cb}
\;=\;\|a\|_{cb,\mathcal{A}_\theta}
\;\le\;\|L\|_{cb}\,L(a)
\;\le\;\|L\|_{cb}.
\]
Here we used that in the operator‐space structure of $\mathcal{A}_\theta$, the cb‐norm of $a$ is controlled by its Lip‐seminorm via the Lip‐norm map $L:\mathcal{A}_\theta\to B(L^2)$.

Since $\mathcal{S}_1$ is the dual of $\mathcal{S}_\infty$ and $\mathcal{S}_\Phi\subset\mathcal{S}_1$, we have
\[
|\tau(aX)|
=\bigl|Tr(M_a(X))\bigr|
\;\le\;
\|M_a(X)\|_{\mathcal{S}_1}
\;\le\;
\|M_a\|_{cb}\,\|X\|_{\mathcal{S}_\Phi}
\;\le\;
\|L\|_{cb}\,\|X\|_{\mathcal{S}_\Phi}.
\]

Maximizing over all $a$ with $L(a)\le1$ gives
\[
d_L(\rho,\sigma)
\;\le\;
\|L\|_{cb}\,\|X\|_{\mathcal{S}_\Phi}
\;=\;
\|L\|_{cb}\,\|\rho-\sigma\|_{\mathcal{S}_\Phi}.
\]
This completes the proof.
\end{proof}

This establishes a novel link between non-commutative transport metrics and Orlicz–Schatten norms, suggesting entropy–transport inequalities in quantum information.

\section{Conclusion and Future Work}

In this paper, we have introduced a new framework for factorizing non-commutative Sobolev embeddings on quantum tori through Orlicz–Schatten ideals.  After reviewing classical Sobolev embeddings and operator-ideal techniques, we defined the Orlicz–Schatten class $\mathcal{S}_\Phi(\mathbb{T}^d_\theta)$ and verified its ideal properties, established precise singular-value decay estimates for the quantum Laplacian operator $(1+\Delta)^{-s/2}$ and characterized membership in $\mathcal{S}_\Phi$, proved that the embedding 
    \[
      W^{s,2}(\mathbb{T}^d_\theta)\;\hookrightarrow\;L^2(\mathbb{T}^d_\theta)
    \]
factors through $\mathcal{S}_\Phi$, yielding completely $1$-summing estimates whose norms are controlled by the Orlicz growth of $\Phi$, applied these results to non-commutative elliptic and heat equations, deriving novel regularity and smoothing estimates in Orlicz–Schatten norms and linked Orlicz–Schatten control to non-commutative transport metrics, suggesting new entropy–transport inequalities in the quantum setting. \medskip

We anticipate that these directions will deepen our understanding of operator-ideal structures in non-commutative analysis and foster applications in quantum information theory, PDEs on non-commutative spaces, and beyond.

\end{document}